\documentclass[a4paper,12pt]{article}
\usepackage[top=2.5cm,bottom=2.5cm,left=2.5cm,right=2.5cm]{geometry}
\usepackage{cite, amsmath, amssymb}
\usepackage{amssymb}
\usepackage{graphicx}
\newenvironment{proof}{{\noindent \it Proof.}}{\hfill $\blacksquare$\par}
\usepackage{latexsym}
\usepackage{graphicx,booktabs,multirow}
\usepackage{tikz}
\usepackage{color}
\usetikzlibrary{decorations.pathreplacing}
\usetikzlibrary{intersections}
\usepackage{enumerate}
\usepackage{graphicx,booktabs,multirow}
\usepackage{appendix}
\newtheorem{theorem}{Theorem}[section]

\newtheorem{lemma}[theorem]{Lemma}

\begin{document}


\title{Maximum   zeroth-order general Randi\'{c} index of orientations of cacti}
\author{ Jiaxiang Yang$^{1}$, Hanyuan Deng$^{1}$\thanks{Corresponding author: hydeng@hunnu.edu.cn}, Zikai Tang$^{1}$, Hechao Liu$^{2}$ \\
{$^{1}$\small Key Laboratory of Computing and Stochastic Mathematics (Ministry of Education),}
 \\{\small School of Mathematics and Statistics, Hunan Normal University,}
 \\{\small  Changsha, Hunan 410081, P. R. China.}
\\{$^{2}$\small School of Mathematical Sciences, South China Normal University,}
 \\{\small  Guangzhou, 510631, P. R. China.}}

\date{}
\maketitle

\begin{abstract}
The  zeroth-order general Randi\'{c} index $R^{0}_{a+1}$ of an $n$-vertices oriented graph $D$  is equal to the sum of  $(d^{+}_{u_i})^{a}+(d^{-}_{u_j})^{a}$ over all arcs $u_iu_j$ of $D$,
where we denote by $d^{+}_{u_i}$   the out-degree  of the vertex $u_i$ and $d^{-}_{u_j}$  the  in-degree of the vertex $u_j$, $a$  is an arbitrary real number. In the paper, we determine the orientations of cacti with the maximum value of the zeroth-order general Randi\'{c} index for $a\geq 1$. 
\end{abstract}

\maketitle

\makeatletter
\renewcommand\@makefnmark%
{\mbox{\textsuperscript{\normalfont\@thefnmark)}}}
\makeatother

\baselineskip=0.25in

\section{Introduction}
\label{1}

For a connected simple graph $G$ with  vertex set $V(G)$ and  edge set $E(G)$,
$N_{G}(u)$ denotes the set of all neighbors of $u$, and we denote by
$d_G(u)=|N_{G}(u)|$  the degree of $u$  in $G$. If $d_G(u)=1$, then the vertex $u$ is  a pendent vertex, and  $uv\in E(G)$ is
 a pendent edge.
Topological index plays an important role  in QSPR
and QSAR; see \cite{1,2}.
 Li and Zheng \cite{3} proposed first general Zagreb index.
In the current study, most scholars call it  the zeroth-order general Randi\'{c} index which is defined as
$R^{0}_{a}(G)=\sum\limits_{u\in V(G)}(d_{G}(u))^{a}=\sum\limits_{uv\in E(G)}[(d_{G}(u))^{a-1}+(d_{G}(v))^{a-1}],$
where  $a$ is  an arbitrary real number.
So, the zeroth-order general Randi\'{c} index is also a VDB topological index.
Over the past few decades, a great deal of study of researchers have shown  many results of extremal  zeroth-order general Randi\'{c} index  of a  graph family; see  \cite {4,5,6,7,8,9,10,11}.
 Monsalve and Rada  gave the defination of  VDB topological indices over orientations of a graph in \cite {JM211}, and brought us some help to study extremal VDB topological indices  of orientations of a  graph   in \cite {JM212}.

For $u, v\in V(G)$,  we can add
an edge $uv\notin E(G)$  to graph $G$  and get the graph $G+uv$.
For $u, v\in V(G)$,  we can delete
an edge $uv\in E(G)$  and get the graph $G-uv$.
If we  delete  $u \in V(G)$ and all $e\in \{e|e=uv\in E(G), v\in V(G)\}$
, then we get the  graph $G-u$.

If a connected graph $G$ only consists of blocks which are either edges or cycles, then $G$ is called  a cactus.   Let $\delta(G)=\min\{d(u)| u\in V(G)\}$.

For a simple  digraph $D$, the digraph $D$ has  vertex set $V(D)$  and arc set $A(D)$, where $uv\in A(D)$  is an arc that from vertex $u$  to vertex $v$ in the graph $D$. For vertex $u \in V(D)$,  the out-degree $d^{+}_u=|\{v| uv\in A(D), v \in V(D)\}|$ (resp. in-degree $d^{-}_u=|\{v| vu\in A(D), v \in V(D)\}|$).

  For a vertex $w\in \{w| d^{+}_w=0,w\in V(D)\}$,  we call  $w$   a sink vertex in $D$.  For a vertex $w\in \{w| d^{-}_w=0,w\in V(D)\}$,  we call  $w$   a  source vertex  in $D$. For a vertex $w\in \{w| d^{+}_w=d^{-}_w=0,w\in V(D)\}$,  we call  $w$    an isolated vertex in $D$. For each $uv\in E(G)$, we use $uv\in A$ or $vu\in A$ to replace $uv\in E(G)$ and get an oriented graph $D$ (i.e. an orientation of $G$ ).
  An orientation of $G$ is called a sink-source orientation of $G$ if
   $d^{+}_w=0$ or $d^{-}_w=0$  for any $w\in V(G)$.  Let $\mathcal{O}(G)=\{D| D$ is a orientation of $G\}$.

The definition  of
the VDB topological indices of a digraph $D$ is as \cite{JM211}
$$\varphi(D)=\frac{1}{2}\sum_{uv\in A}\varphi(d^{+}_{u}, d^{-}_v),$$
where $\varphi(d^{+}_{u}, d^{-}_v)$ is  an arbitrary real number.
Specially, if $\varphi(i, j)=i^{a-1}+j^{a-1}$,  then $\varphi(D)$ is the zeroth-order general Randi\'{c} index  of a digraph, denoted by $R^{0}_{a}(D)$, i.e.,
$$R^{0}_{a}(D)=\frac{1}{2}\sum\limits_{uv\in A}[(d^{+}_u)^{a-1}+(d^{-}_v)^{a-1}].$$


Monsalve and Rada \cite {JM211} determined sharp upper and lower bounds of the Randi\'{c} index over the set of all digraphs in terms of the order,  the set of all  oriented trees in terms of the order, the set of all oriented paths and cycles in terms of the order,   the set of all oriented  hypercubes in terms of the  dimension, respectively.
Deng et al. \cite{HD22} determined sharp upper and lower bounds of some VDB topological indices
 over the set of all digraphs in terms of the order, such as the harmonic index, the Atom-Bond-Connectivity index, the Geometric-Arithmetic index, the sum-connectivity index, the first and second Zagreb indices and the Randi\'{c} index.
The maximal values and minimal values  of  VDB topological
indices over the set of  all oriented trees in terms of the order and the set of all oriented graphs in terms of the order were obtained by  Monsalve and Rada \cite {JM212}.
 Soon after,  the problem that is to determine a sharp upper  bound of first Zagreb  index over the set of all oriented  unicyclic graphs in terms of the order and
matching number  was solved by  Yang and  Deng \cite{JY22}.

 In this paper, we discuss orientations of cacti in terms of the order with maximum value  of  the zeroth-order general Randi\'{c} index.

\section{Some Lemmas}

 A lot of lemmas  correlated with main result are given in the section. Let $\mathcal{G}(n, r)$ be the set of cacti of order $n$ and  $r$ cycles. Let $G^{0}(n, r)$, $G^{(1)}(n,r), G^{(2)}(n,r)$ be the graphs  depicted in Figure \ref{fig-1}. $\tilde{\mathcal{G}}(n,r)=\{G^{(1)}(n,r), G^{(2)}(n,r)\}$.
 We denote by $C_n$  the the cycle which consists of $n$ vertices, and $G^{(3)}(4,1)$, $G^{(4)}(4,1)$ the sink-sourse orientations of $C_4$.  Let  $\mathcal{G}^{*}(n,r)= \begin{cases}\tilde{\mathcal{G}}(n,r)\bigcup \{G^{(3)}(4,1), G^{(4)}(4,1)\}, & \text { if } a=1 ~and ~(n,r)=(4,1) \\ \tilde{\mathcal{G}}(n,r), & \text { otherwise}  \end{cases}.$

\begin{figure}[ht]
\begin{center}
  \includegraphics[width=10cm,height=3cm]{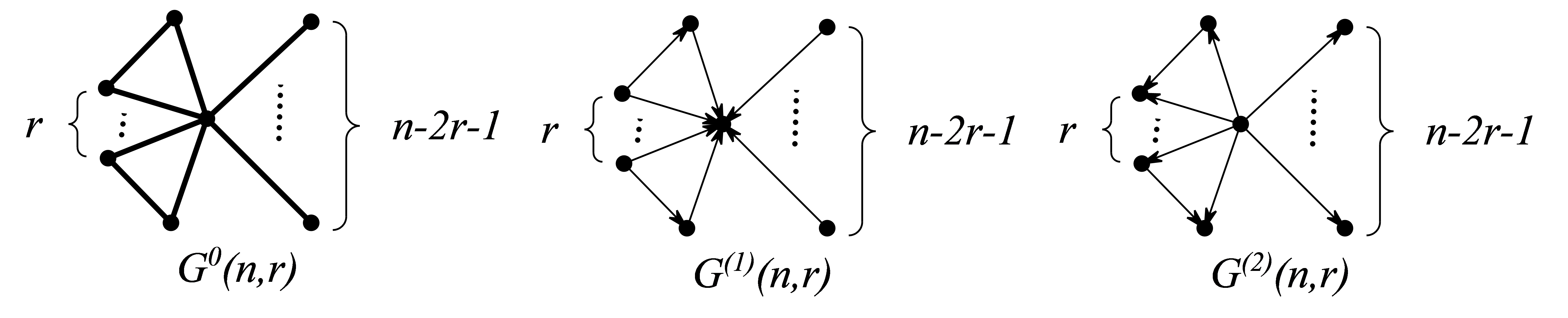}
 \end{center}
\vskip -0.5cm
\caption{$G^{0}(n,r)$ and its two orientations}\label{fig-1}
\end{figure}

\begin{lemma}\cite {JM19}\label{lem1}
  If $G$ is a connected  bipartite graph, then $G$ has two sink-source orientations.
\end{lemma}

\begin{lemma}\label{lem2}
For any simple connected  graph $G$,
 $D \in \mathcal{O}(G)$, $a\geq 1$, we have
$$R^{0}_{a+1}(D)\leq \frac{R^{0}_{a+1}(G)}{2}$$  with equality  if and only if $D$ is a sink-source orientation of $G$.

\end{lemma}
\begin{proof}
 $R^{0}_{a+1}(D)=\frac{1}{2}\sum\limits_{uv\in A(D)}[(d^{+}_{D}(u))^a+(d^{-}_{D}(v))^a]$, $R^{0}_{a+1}(G)=\sum\limits_{uv\in E(G)}[(d_{G}(u))^a+(d_{G}(v))^a]$. And for any $u\in V(G)$, we have $d_{G}(u)\geq \max\{d^{+}_{D}(u), d^{-}_{D}(u)\}$. So, $$R^{0}_{a+1}(D)=\frac{1}{2}\sum\limits_{uv\in A(D)}[(d^{+}_{D}(u))^a+(d^{-}_{D}(v))^a]\leq \frac{1}{2}\sum\limits_{uv\in E(G)}[(d_{G}(u))^a+(d_{G}(v))^a]= \frac{1}{2}R^{0}_{a+1}(G)$$
 with equality  if and only if $d_{G}(u)= \max\{d^{+}_{D}(u), d^{-}_{D}(u)\}$ for any $u\in V(G)$ which implies  that $D$ is a  sink-source orientation of $G$.
\end{proof}
\begin{lemma}\label{lem4}
 Let $G\in \mathcal{G}(n,r)$, $n\geq 6$, $r\geq 2$, $\delta(G)\geq 2$.  If there exists $u_0u_1\in E(G)$ such that
 $d_{G}(u_0)=d_{G}(u_1)=2$, $\{u_2\}=N_{G}(u_0)\backslash \{u_1\}$ and $d_{G}(u_2)=d \geq 3$.  $u_1u_2\notin E(G)$, $G'=G-u_0+u_1u_2$, $G''=G-u_0$,  $D \in \mathcal{O}(G)$, $D' \in \mathcal{O}(G')$, $D'' \in \mathcal{O}(G'')$, such that  $A(D')\cap A(D)=A(D'')$, $a\geq 1$, then  $$R^{0}_{a+1}(D)-R^{0}_{a+1}(D')\leq \frac{1}{2}[2^{a+1}-1+d^{a+1}-(d-1)^{a+1}]$$
    with equality  if and only if $u_1u_0, u_2u_0\in A(D), u_1u_2\in A(D'), d^{+}_{D}(u_2)=d_{G}(u_2)$; or  $u_0u_1, u_0u_2\in A(D), u_2u_1\in A(D'), d^{-}_{D}(u_2)=d_{G}(u_2)$; or $u_0u_1, u_2u_0\in A(D), u_1u_2\in A(D'),d^{-}_{D}(u_1)=d_{G}(u_1), d^{+}_{D}(u_2)=d_{G}(u_2)$; or $u_1u_0, u_0u_2\in A(D), u_2u_1\in A(D'), d^{+}_{D}(u_1)=d_{G}(u_1), d^{-}_{D}(u_2)=d_{G}(u_2)$.
\end{lemma}

\begin{proof}
Let $N_{G}(u_1)\backslash \{u_0\}=\{w_1\}$, $N_{G}(u_2)\backslash \{u_0\}=\{v_1,v_2,\cdot\cdot\cdot,v_{d_{G}(u_2)-1} \}$. It is straightforward to check  that $d^{+}_{D}(w_1)=d^{+}_{D'}(w_1)$, $d^{-}_{D}(w_1)=d^{-}_{D'}(w_1)$; $d^{+}_{D}(v_i)=d^{+}_{D'}(v_i)$, $d^{-}_{D}(v_i)=d^{-}_{D'}(v_i)$ $(i=1,2,\cdot\cdot\cdot,d_{G}(u_2)-1)$.
Let $p_i=d^{+}_{D}(u_i)$, $i=0,1,2$; $q_i=d^{-}_{D}(u_i)$, $i=0,1,2$; $p'_{i}=d^{+}_{D'}(u_i)$, $i=0,1,2$; $q'_{i}=d^{-}_{D'}(u_i)$, $i=0,1,2$.
We distinguish the following eight cases to prove the lemma.

\textbf{Case 1}. $u_0u_1, u_0u_2\in A(D)$, $u_1u_2\in A(D')$.

 For $u_0u_1, u_0u_2\in A(D)$, we have  that
 $p_0=p'_0+2=0+2=2$, $q_0=q'_0=0$;
  For $u_0u_1\in A(D)$, $u_1u_2\in A(D')$, we have  that
  $p_1=p'_1-1$,
 $q_1=q'_1+1$;
 For $u_0u_2\in A(D)$, $u_1u_2\in A(D')$, we have  that
  $q_2=q'_2$, $p_2=p'_2$.
Hence,
\begin{equation*}
\begin{aligned}
R^{0}_{a+1}(D)-R^{0}_{a+1}(D')=& \frac{1}{2}[\varphi(p_0, q_1)+\varphi(p_0, q_2)-\varphi(p'_1, q'_2)+\sum^{q_1-1}_{i=1}[\varphi(q_1, d^{+}_{D}(w_i))\\
&-\varphi(q'_1, d^{+}_{D'}(w_i))]+\sum^{d_{G}(u_1)-1}_{i=q_1}[\varphi(p_1, d^{-}_{D}(w_i))-\varphi(p'_1, d^{-}_{D'}(w_i))]\\
&+\sum^{d_{G}(u_2)-1}_{i=q_2}[\varphi(p_2, d^{-}_{D}(v_i))-\varphi(p'_2, d^{-}_{D'}(v_i))]+\sum^{q_2-1}_{i=1}[\varphi(q_2,d^{+}_{D}(v_i))\\
&-\varphi(q'_2, d^{+}_{D'}(v_i))]]\\
=& \frac{1}{2}[(p_0)^{a}+(q_1)^{a}+(p_0)^{a}+(q_2)^{a}-(p_1+1)^{a}-(q_2)^{a}+(q_1-1)\\
&[(q_1)^{a}-(q_1-1)^{a}]+(q_2-1)[(q_2)^{a}-(q_2)^{a}]]\\
\leq& \frac{1}{2}[(p_0)^{a}+(d_{G}(u_1))^{a}+(p_0)^{a}-1+(d_{G}(u_1)-1)[(d_{G}(u_1))^{a}\\
&-(d_{G}(u_1)-1)^{a}]]\\
\leq& \frac{1}{2}[2^a+2^a+2^a-1+2^a-1] \\
=& 2^{a+1}-1
\end{aligned}
\end{equation*}
Let $f(d)=\frac{1}{2}[2^{a+1}-1+d^{a+1}-(d-1)^{a+1}]-(2^{a+1}-1)=\frac{1}{2}[d^{a+1}-(d-1)^{a+1}+1-2^{a+1}]$, then $f'(d)=\frac{1}{2}[(a+1)d^{a}-(a+1)(d-1)^{a}]>0$. Hence $\min\limits_{d\geq 3} f(d)=f(3)=\frac{1}{2}[3^{a+1}-2^{a+2}+1]\geq \frac{1}{2}[9\times 3^{a-1}-8\times 2^{a-1}]>0 $, we have
$$R^{0}_{a+1}(D)-R^{0}_{a+1}(D')\leq  2^{a+1}-1< \frac{1}{2}[2^{a+1}-1+d^{a+1}-(d-1)^{a+1}].$$

\textbf{Case 2}. $u_0u_1, u_2u_0\in A(D)$, $u_1u_2\in A(D')$.

For $u_0u_1, u_2u_0\in A(D)$, we have  that
 $p_0=p'_0+1=0+1=1$, $q_0=q'_0+1=0+1=1$;
For $u_0u_1 \in A(D)$,$u_1u_2\in A(D')$, we have  that
 $p_1=p'_1-1$,  $q_1=q'_1+1$;
For $u_2u_0 \in A(D)$,$u_1u_2\in A(D')$, we have  that
 $p_2=p'_2+1$,
 $q_2=q'_2-1$.
Hence,
\begin{equation*}
\begin{aligned}
R^{0}_{a+1}(D)-R^{0}_{a+1}(D')=& \frac{1}{2}[(p_0)^{a}+(q_1)^{a}+(p_2)^{a}+(q_0)^{a}-(p_1+1)^{a}\\
&-(q_2+1)^{a}+(q_1-1)[(q_1)^{a}-(q_1-1)^{a}]\\
&+p_1[(p_1)^{a}-(p_1+1)^{a}]+(p_2-1)[(p_2)^{a}\\
&-(p_2-1)^{a}]+q_2[(q_2)^{a}-(q_2+1)^{a}]]\\
\leq& \frac{1}{2}[1+(d_{G}(u_1))^{a}+(d_{G}(u_2))^{a}+1-1-1\\
&+(d_{G}(u_1)-1)[(d_{G}(u_1))^{a}-(d_{G}(u_1)-1)^{a}]\\
&+(d_{G}(u_2)-1)[(d_{G}(u_2))^{a}-(d_{G}(u_2)-1)^{a}]]\\
=& \frac{1}{2}[2^{a+1}-1+d^{a+1}-(d-1)^{a+1}]
\end{aligned}
\end{equation*}
 with equality  if and only if $q_1=d_{G}(u_1)$, $p_2=d_{G}(u_2)$.

\textbf{Case 3}. $u_1u_0, u_0u_2\in A(D)$, $u_1u_2\in A(D')$.

For  $u_1u_0, u_0u_2\in A(D)$, we have  that
 $p_0=p'_0+1=0+1=1$, $q_0=q'_0+1=0+1=1$;
For  $u_1u_0\in A(D)$, $u_1u_2\in A(D')$, we have  that
  $q_1=q'_1$, $p_1=p'_1$;
For  $u_0u_2\in A(D)$, $u_1u_2\in A(D')$, we have  that
 $p_2=p'_2$,
  $q_2=q'_2$.
Hence,
\begin{equation*}
\begin{aligned}
R^{0}_{a+1}(D)-R^{0}_{a+1}(D')=& \frac{1}{2}[(p_1)^{a}+(q_0)^{a}+(p_0)^{a}+(q_2)^{a}-(p'_1)^{a}-(q'_2)^{a}]\\
=& \frac{1}{2}[(q_0)^{a}+(p_0)^{a}] \\
=& 1< \frac{1}{2}[2^{a+1}-1+d^{a+1}-(d-1)^{a+1}]
\end{aligned}
\end{equation*}

\textbf{Case 4}. $u_1u_0, u_2u_0\in A(D)$, $u_1u_2\in A(D')$.

For $u_1u_0, u_2u_0\in A(D)$, we have  that
 $p_0=p'_0=0$, $q_0=q'_0+2=0+2=2$;
For $u_1u_0\in A(D)$, $u_1u_2\in A(D')$, we have  that
 $p_1=p'_1$, $q_1=q'_1$;
For $u_2u_0\in A(D)$, $u_1u_2\in A(D')$, we have  that
 $p_2=p'_2+1$,
  $q_2=q'_2-1$.
Hence,
\begin{equation*}
\begin{aligned}
R^{0}_{a+1}(D)-R^{0}_{a+1}(D')=& \frac{1}{2}[(p_1)^{a}+(q_0)^{a}+(p_2)^{a}+(q_0)^{a}-(p_1)^{a}-(q_2+1)^{a}+(p_2-1)[(p_2)^{a}\\
&-(p_2-1)^{a}]+q_2[(q_2)^{a}-(q_2+1)^{a}]]\\
\leq& \frac{1}{2}[2^{a}+(d_{G}(u_2))^{a}+2^{a}-1+(d_{G}(u_2)-1)[(d_{G}(u_2))^{a}\\
&-(d_{G}(u_2)-1)^{a}]]\\
=& \frac{1}{2}[2^a+d^a+2^a-1+(d-1)(d^a-(d-1)^a)] \\
=& \frac{1}{2}[2^{a+1}-1+d^{a+1}-(d-1)^{a+1}]
\end{aligned}
\end{equation*}
 with equality if and only if $p_2=d_{G}(u_2)$.

\textbf{Case 5}. $u_0u_1, u_0u_2\in A(D)$, $u_2u_1\in A(D')$.

Similarly to \textbf{Case 4}, we have $$R^{0}_{a+1}(D)-R^{0}_{a+1}(D')\leq \frac{1}{2}[2^{a+1}-1+d^{a+1}-(d-1)^{a+1}]$$
with equality if and only if $q_2=d_{G}(u_2)$.

\textbf{Case 6}. $u_0u_1, u_2u_0\in A(D)$, $u_2u_1\in A(D')$.

Similarly to \textbf{Case 3}, we have $$R^{0}_{a+1}(D)-R^{0}_{a+1}(D')< \frac{1}{2}[2^{a+1}-1+d^{a+1}-(d-1)^{a+1}].$$

\textbf{Case 7}. $u_1u_0, u_0u_2\in A(D)$, $u_2u_1\in A(D')$.

Similarly to \textbf{Case 2}, we have $$R^{0}_{a+1}(D)-R^{0}_{a+1}(D')\leq \frac{1}{2}[2^{a+1}-1+d^{a+1}-(d-1)^{a+1}]$$
  with equality if and only if $p_1=d_{G}(u_1)$, $q_2=d_{G}(u_2)$.

\textbf{Case 8}. $u_1u_0, u_2u_0\in A(D)$, $u_2u_1\in A(D')$.

Similarly to \textbf{Case 1}, we have $$R^{0}_{a+1}(D)-R^{0}_{a+1}(D')<\frac{1}{2}[2^{a+1}-1+d^{a+1}-(d-1)^{a+1}].$$

Consequently, $$R^{0}_{a+1}(D)-R^{0}_{a+1}(D')\leq \frac{1}{2}[2^{a+1}-1+d^{a+1}-(d-1)^{a+1}]$$
    with equality if and only if $u_1u_0, u_2u_0\in A(D), u_1u_2\in A(D'), d^{+}_{D}(u_2)=d_{G}(u_2)$; or  $u_0u_1, u_0u_2\in A(D), u_2u_1\in A(D'), d^{-}_{D}(u_2)=d_{G}(u_2)$; or $u_0u_1, u_2u_0\in A(D), u_1u_2\in A(D'), d^{-}_{D}(u_1)=d_{G}(u_1), d^{+}_{D}(u_2)=d_{G}(u_2)$; or $u_1u_0, u_0u_2\in A(D), u_2u_1\in A(D'), d^{+}_{D}(u_1)=d_{G}(u_1), d^{-}_{D}(u_2)=d_{G}(u_2)$.
\end{proof}

\begin{lemma}\label{lem5}
Let $G\in \mathcal{G}(n,r)$ with $n\geq 6$
and $r\geq 2$, $\delta(G)\geq 2$. If there exists $u_0u_1\in E(G)$ such that
 $d_{G}(u_0)=d_{G}(u_1)=2$, $\{u_2\}=N_{G}(u_0)\backslash \{u_1\}$ and $d_{G}(u_2)=d\geq 3$. $u_1u_2\in E(G)$, $G'=G-u_0-u_1$, $D \in \mathcal{O}(G)$, $D' \in \mathcal{O}(G')$ such that $A(D')\cap A(D)=A(D')$, $a\geq 1$,  then  $$R^{0}_{a+1}(D)-R^{0}_{a+1}(D')\leq
 \frac{1}{2}[2^{a+1}+2+d^{a+1}-(d-2)^{a+1}]$$
    with equality if and only if $d^{-}_{D}(u_2)=d_{G}(u_2)$  or $d^{+}_{D}(u_2)=d_{G}(u_2)$.

\end{lemma}
\begin{proof}
Let $N_{G}(u_2)\backslash \{u_0,u_1\}=\{v_1,v_2,\cdot\cdot\cdot,v_{d_{G}(u_2)-2} \}$. It is obvious that $d^{+}_{D}(v_i)=d^{+}_{D'}(v_i)$, $d^{-}_{D}(v_i)=d^{-}_{D'}(v_i)$ $(i=1,2,\cdot\cdot\cdot,d_{G}(u_2)-2)$.
Let $p_i=d^{+}_{D}(u_i)$, $i=0,1,2$; $q_i=d^{-}_{D}(u_i)$, $i=0,1,2$; $p'_{i}=d^{+}_{D'}(u_i)$, $i=0,1,2$; $q'_{i}=d^{-}_{D'}(u_i)$, $i=0,1,2$.
According to $u_0,u_1,u_2$, we have eight cases.

\textbf{Case 1}. $u_0u_1, u_0u_2\in A(D)$, $u_1u_2\in A(D)$.

For $u_0u_1, u_0u_2\in A(D)$, we have  that
 $p_0=p'_0+2=2$, $q_0=q'_0=0$;
For $u_0u_1\in A(D)$, $u_1u_2\in A(D)$, we have  that
 $p_1=p'_1+1=0+1=1$, $q_1=q'_1+1=0+1=1$;
For $ u_0u_2\in A(D)$, $u_1u_2\in A(D)$, we have  that
 $p_2=p'_2$,
  $q_2=q'_2+2$.
Hence,

\begin{equation*}
\begin{aligned}
R^{0}_{a+1}(D)-R^{0}_{a+1}(D')=& \frac{1}{2}[\varphi(p_0,q_1)+\varphi(p_0,q_2)+\varphi(p_1,q_2)+\sum^{q_2-2}_{i=1}[\varphi(q_2,d^{+}_{D}(v_i))\\
&-\varphi(q'_2,d^{+}_{D'}(v_i))]+\sum^{d_{G}(u_2)-2}_{i=q_2-1}[\varphi(p_2,d^{-}_{D}(v_i))-\varphi(p'_2,d^{-}_{D'}(v_i))]]\\
=& \frac{1}{2}[(p_0)^{a}+(q_1)^{a}+(p_0)^{a}+(q_2)^{a}+(p_1)^{a}+(q_2)^{a}+(q_2-2)\\
&[(q_2)^{a}-(q_2-2)^{a}]]\\
\leq& \frac{1}{2}[2^{a}+1^{a}+2^{a}+(d_{G}(u_2))^{a}+1^{a}+(d_{G}(u_2))^{a}\\
&+(d_{G}(u_2)-2)[(d_{G}(u_2))^{a}-(d_{G}(u_2)-2)^{a}]]\\
=& \frac{1}{2}[2^{a+1}+2+d^a+d^a+(d-2)(d^a-(d-2)^a)] \\
=& \frac{1}{2}[2^{a+1}+2+d^{a+1}-(d-2)^{a+1}]
\end{aligned}
\end{equation*}
 with equality if and only if $q_2=d_{G}(u_2)$.

\textbf{Case 2}. $u_0u_1, u_2u_0\in A(D)$, $u_1u_2\in A(D)$.

For $u_0u_1, u_2u_0\in A(D)$, we have  that
 $p_0=p'_0+1=0+1=1$, $q_0=q'_0+1=0+1=1$;
For $u_0u_1\in A(D)$, $u_1u_2\in A(D)$, we have  that
 $p_1=p'_1+1=0+1=1$, $q_1=q'_1+1=0+1=1$;
For $ u_2u_0\in A(D)$, $u_1u_2\in A(D)$, we have  that
  $p_2=p'_2+1$,
  $q_2=q'_2+1$.
Hence,
\begin{equation*}
\begin{aligned}
R^{0}_{a+1}(D)-R^{0}_{a+1}(D')=& \frac{1}{2}[(p_0)^{a}+(q_1)^{a}+(p_2)^{a}+(q_0)^{a}+(p_1)^{a}+(q_2)^{a}+(q_2-1)[(q_2)^{a}\\
&-(q'_2)^{a}]+(p_2-1)[(p_2)^{a}-(p'_2)^{a}]]\\
\leq& \frac{1}{2}[4+(q_2)^{a+1}-(q_2-1)^{a+1}+(p_2)^{a+1}-(p_2-1)^{a+1}]
\end{aligned}
\end{equation*}
 Let $f(x)=\frac{1}{2}[4+x^{a+1}-(x-1)^{a+1}+(d-x)^{a+1}-(d-x-1)^{a+1}]$ and $g(x)=(a+1)x^{a}-(a+1)(x-1)^{a}$ $(1\leq x\leq d-1)$, then $f'(x)=\frac{1}{2}[(a+1)x^{a}-(a+1)(x-1)^{a}-[(a+1)(d-x)^{a}-(a+1)(d-x-1)^{a}]]=
 \frac{1}{2}[g(x)-g(d-x)]$.
By $g'(x)=a(a+1)x^{a-1}-a(a+1)(x-1)^{a-1}\geq 0$,
we have  $f'(x)\geq 0$ for $x\geq d-x$, and $f'(x)\leq 0$ for $x\leq d-x$.
It is obvious that $\max\limits_{1\leq x\leq d-1} f(x)=f(d-1)=f(1)=\frac{1}{2}[4+(d-1)^{a+1}-(d-2)^{a+1}+1]$.
Since
\begin{equation*}
\begin{aligned} &\frac{1}{2}[4+(d-1)^{a+1}-(d-2)^{a+1}+1]-\frac{1}{2}[d^{a+1}-(d-2)^{a+1}+2+2^{a+1}]\\
& = \frac{1}{2}[3+(d-1)^{a+1}-d^{a+1}-2^{a+1}]\\
&<0,
\end{aligned}
\end{equation*}
 $R^{0}_{a+1}(D)-R^{0}_{a+1}(D')<\frac{1}{2}[2^{a+1}+2+d^{a+1}-(d-2)^{a+1}]$.

\textbf{Case 3}. $u_1u_0, u_0u_2\in A(D)$, $u_1u_2\in A(D)$.

For $u_1u_0, u_0u_2\in A(D)$, we have  that
 $p_0=p'_0+1=0+1=1$, $q_0=q'_0+1=0+1=1$;
For $u_1u_0\in A(D)$, $u_1u_2\in A(D)$, we have  that
 $p_1=p'_1+2=0+2=2$, $q_1=q'_1=0$;
For $ u_0u_2\in A(D)$, $u_1u_2\in A(D)$, we have  that
 $p_2=p'_2$,
  $q_2=q'_2+2$.
Hence,
\begin{equation*}
\begin{aligned}
R^{0}_{a+1}(D)-R^{0}_{a+1}(D')=& \frac{1}{2}[(p_1)^{a}+(q_0)^{a}+(p_0)^{a}+(q_2)^{a}+(p_1)^{a}+(q_2)^{a}+(q_2-2)[(q_2)^{a}\\
&-(q_2-2)^{a}]]\\
\leq& \frac{1}{2}[2^{a}+1^{a}+1^{a}+(d_{G}(u_2))^{a}+2^{a}+(d_{G}(u_2))^{a}+(d_{G}(u_2)-2)[(d_{G}(u_2))^{a}\\
&-(d_{G}(u_2)-2)^{a}]]\\
=& \frac{1}{2}[2^{a}+1+1+d^a+2^a+d^a+(d-2)(d^a-(d-2)^a)] \\
=& \frac{1}{2}[2^{a+1}+2+d^{a+1}-(d-2)^{a+1}]
\end{aligned}
\end{equation*}
 with equality if and only if $q_2=d_{G}(u_2)$.

\textbf{Case 4}. $u_1u_0, u_2u_0\in A(D)$, $u_1u_2\in A(D)$.

For $u_1u_0, u_2u_0\in A(D)$, we have  that
 $p_0=p'_0=0$, $q_0=q'_0+2=0+2=2$;
For $u_1u_0\in A(D)$, $u_1u_2\in A(D)$, we have  that
  $p_1=p'_1+2=0+2=2$, $q_1=q'_1=0$;
For $ u_2u_0\in A(D)$, $u_1u_2\in A(D)$, we have  that
 $p_2=p'_2+1$,
  $q_2=q'_2+1$.
Hence, we have the inequality by the proof of \textbf{Case 2}.
\begin{equation*}
\begin{aligned}
R^{0}_{a+1}(D)-R^{0}_{a+1}(D')=& \frac{1}{2}[(p_1)^{a}+(q_0)^{a}+(q_0)^{a}+(p_2)^{a}+(p_1)^{a}+(q_2)^{a}+(q_2-1)\\
&[(q_2)^{a}-(q_2-1)^{a}]+(p_2-1)[(p_2)^{a}-(p_2-1)^{a}]]\\
=& \frac{1}{2}[2^{a+2}+(q_2)^{a+1}-(q_2-1)^{a+1}+(p_2)^{a+1}-(p_2-1)^{a+1}]\\
\leq& \frac{1}{2}[2^{a+2}+(d-1)^{a+1}-(d-2)^{a+1}+1]
\end{aligned}
\end{equation*}
Let  $h(d)=\frac{1}{2}[2^{a+2}+(d-1)^{a+1}-(d-2)^{a+1}+1]-\frac{1}{2}[2^{a+1}+2+d^{a+1}-(d-2)^{a+1}]= \frac{1}{2}[2^{a+1}+(d-1)^{a+1}-d^{a+1}-1]$, we have $h'(d)=\frac{1}{2}[(a+1)(d-1)^{a}-(a+1)d^{a}]< 0$, and   $\max\limits_{d\geq 3}{h(d)}=h(3)=\frac{1}{2}[2^{a+2}-3^{a+1}-1]=\frac{1}{2}[8\times 2^{a-1}-9\times 3^{a-1}-1]< 0$.
Hence, $$R^{0}_{a+1}(D)-R^{0}_{a+1}(D')< \frac{1}{2}[2^{a+1}+2+d^{a+1}-(d-2)^{a+1}].$$

\textbf{Case 5}. $u_0u_1, u_0u_2\in A(D)$, $u_2u_1\in A(D)$.

Similarly to \textbf{Case 4}, we have
$$R^{0}_{a+1}(D)-R^{0}_{a+1}(D')< \frac{1}{2}[2^{a+1}+2+d^{a+1}-(d-2)^{a+1}].$$

\textbf{Case 6}. $u_0u_1, u_2u_0\in A(D)$, $u_2u_1\in A(D)$.

Similarly to \textbf{Case 3}, we have
$$R^{0}_{a+1}(D)-R^{0}_{a+1}(D')\leq   \frac{1}{2}[2^{a+1}+2+d^{a+1}-(d-2)^{a+1}].$$
 with equality if and only if  $p_2=d_{G}(u_2)$.

\textbf{Case 7}. $u_1u_0, u_0u_2\in A(D)$, $u_2u_1\in A(D)$.

Similarly to \textbf{Case 2}, we have
$$R^{0}_{a+1}(D)-R^{0}_{a+1}(D')< \frac{1}{2}[2^{a+1}+2+d^{a+1}-(d-2)^{a+1}].$$

\textbf{Case 8}. $u_1u_0, u_2u_0\in A(D)$, $u_2u_1\in A(D)$.

Similarly to \textbf{Case 1}, we have
$$R^{0}_{a+1}(D)-R^{0}_{a+1}(D')\leq \frac{1}{2}[2^{a+1}+2+d^{a+1}-(d-2)^{a+1}]$$  with equality if and only if $p_2=d_{G}(u_2)$.

Consequently, $$R^{0}_{a+1}(D)-R^{0}_{a+1}(D')\leq \frac{1}{2}[2^{a+1}+2+d^{a+1}-(d-2)^{a+1}]$$
  with equality if and only if $d^{-}_{D}(u_2)=d_{G}(u_2)$  or $d^{+}_{D}(u_2)=d_{G}(u_2)$.
\end{proof}

In the following,  we determine the maximum    zeroth-order general Randi\'{c} index of trees, and the corresponding extreme graph.

\begin{lemma}\cite{4}\label{lem6}
Let $G\in \mathcal{G}(n,0)$ with $n\geq 2$, $a\geq 1$. Then
$$R^{0}_{a+1}(G)\leq n-1+(n-1)^{1+a}$$
 with equality if and only if $G \cong G^{0}(n,0)$.
\end{lemma}

 From the maximum  tree on  the zeroth-order general Randi\'{c} index and the corresponding extremum graph, we get the maximum  oriented  tree with respect to  the zeroth-order general Randi\'{c} index and the corresponding extremum digraph.

\begin{lemma}\label{lem7}
Let $G\in \mathcal{G}(n,0)$ with $n\geq 2$, $D\in \mathcal{O}(G)$, $a\geq 1$. Then $$R^{0}_{a+1}(D)\leq \frac{1}{2}[n-1+(n-1)^{1+a}] $$ with equality if and only if $D\in  \mathcal{G}^{*}(n,0)$.
\end{lemma}

\begin{proof}
     It is straghtforward to check that $G$ is a bipartite graph. By Lemma \ref{lem1} and  Lemma \ref{lem2}, $R^{0}_{a+1}(D)\leq \frac{1}{2}R^{0}_{a+1}(G)$ with equality if and only if $D$ is a sink-source orientation of $G$.\\
     \indent We can obtain  the following equality by Lemma \ref{lem6}, $$\max \{R^{0}_{a+1}(D)|\\D\in \mathcal{O}(G), G\in \mathcal{G}(n,0)\}=\max \{\frac{1}{2}R^{0}_{a+1}(G)| G\in \mathcal{G}(n,0)\}=\frac{1}{2}R^{0}_{a+1}( G^{0}(n,0))$$
    Consequently,  $$R^{0}_{a+1}(D)\leq \frac{1}{2}[n-1+(n-1)^{1+a}]$$
     with equality if and only if $D\in \mathcal{G}^{*}(n,0)$.
\end{proof}

For a unicyclic graph of order $4$ with a triangle, we have

\begin{lemma}\label{lem10}
Let $G_1$ be the graph  depicted in Figure \ref{fig-2}, $D\in \mathcal{O}(G_1)$, $a\geq 1$. Then $$R^{0}_{a+1}(D)\leq \frac{1}{2}[3^{a+1}+2^{a+1}+3]$$
 with equality if and only if  $D\in \{G^{(1)}(4,1),G^{(2)}(4,1)\}$.
\end{lemma}
\begin{proof}
It is straightforward to check that $\mathcal{O}(G_1)=\{D_1,D_2,\cdot\cdot\cdot,D_{6}\}$, where $D_1,D_2,\cdot\cdot\cdot,D_{6}$ are depicted in Figure \ref{fig-2}.  By  direct computing, we have
$$R^{0}_{a+1}(D_5)=R^{0}_{a+1}(D_1)=\frac{1}{2}(3^{a+1}+2^{a+1}+3)$$
$$R^{0}_{a+1}(D_4)=R^{0}_{a+1}(D_2)=\frac{1}{2}(4+2^{a+2})$$
$$R^{0}_{a+1}(D_6)=R^{0}_{a+1}(D_3)=\frac{1}{2}(2+2^{a+2}+2^{a+1}).$$
As the fact that
$$R^{0}_{a+1}(D_1)-R^{0}_{a+1}(D_2)=\frac{1}{2}(3^{a+1}+2^{a+1}+3)-\frac{1}{2}(4+2^{a+2})=\frac{1}{2}(3^{a+1}-2^{a+1}-1)>0,$$

$R^{0}_{a+1}(D_1)-R^{0}_{a+1}(D_3)=\frac{1}{2}(3^{a+1}+2^{a+1}+3)-\frac{1}{2}(2+2^{a+2}+2^{a+1})=\frac{1}{2}(3^{a+1}-2^{a+2}+1)=\frac{1}{2}(9\times 3^{a-1}-8\times 2^{a-1}+1)>0,$
$$R^{0}_{a+1}(D)\leq \frac{1}{2}[3^{a+1}+2^{a+1}+3]$$
 with equality if and only if  $D\in \{D_1, D_5\}= \{G^{(1)}(4,1),G^{(2)}(4,1)\}$.
\end{proof}
\begin{figure}[ht]
\begin{center}
  \includegraphics[width=10cm,height=3cm]{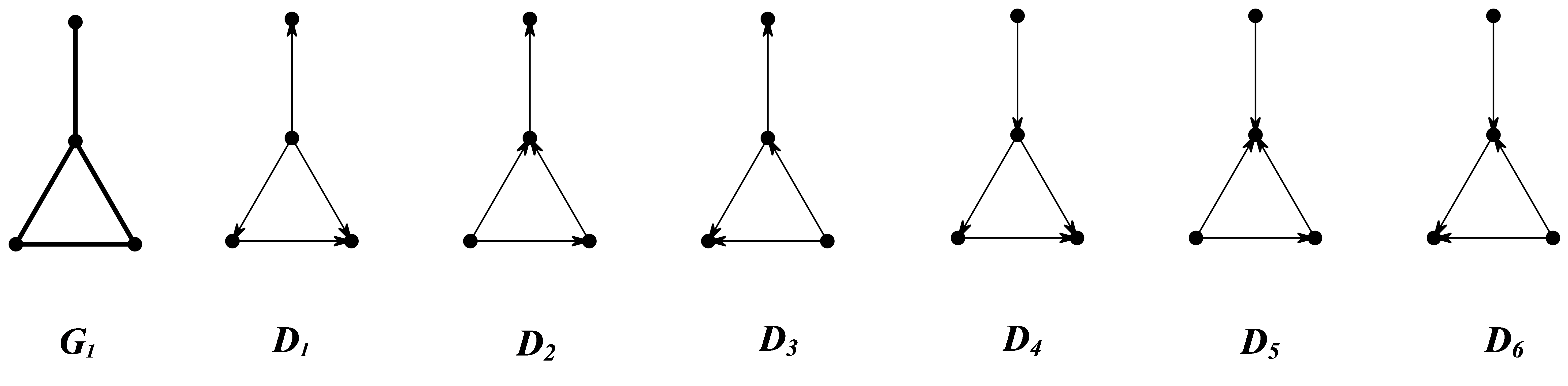}
 \end{center}
\vskip -0.5cm
\caption{$G_1$ and its orientations in Lemma \ref{lem10} }\label{fig-2}
\end{figure}



\begin{lemma}\label{lem102}
Let $G$ be a simple connected graph, $d_{G}(u)=1$ and $u\in V(G)$, $G'=G-u$, $D\in \mathcal{O}(G)$,  $D'\in \mathcal{O}(G')$ such that  $A(D')\bigcap A(D)=A(D')$, $v \in N_{G}(u)$, $a\geq 1$. Then
$$R^{0}_{a+1}(D)-R^{0}_{a+1}(D')\leq \frac{1}{2}[1+(d_{G}(v))^{a+1}-(d_{G}(v)-1)^{a+1}]$$
  with equality if and only if   $\max\{d^{+}_{D}(v), d^{-}_{D}(v)\}=d_{G}(v)$.
\end{lemma}
\begin{proof}
Let  $N_{G}(v)=\{u=u_1, u_2, \cdot\cdot\cdot,u_{d_{G}(v)}\}$ and $d_{D}^{+}(v)=p$, $d_{D}^{-}(v)=q$, $d_{D'}^{+}(v)=p'$, $d_{D'}^{-}(v)=q'$.

\textbf{Case 1}. $uv\in A(D)$

In this case, $q=q'+1$, $p=p'$. It is straightforward to check that  $d_{D}^{+}(u_i)=d_{D'}^{+}(u_i)$,  $d_{D}^{-}(u_i)=d_{D'}^{-}(u_i) $, $i=2,3,\cdot\cdot\cdot, d_{G}(v)$.  Moreover,
\begin{equation*}
\begin{aligned}
R^{0}_{a+1}(D)-R^{0}_{a+1}(D')=& \frac{1}{2}[R^{0}_{a+1}(d^{+}_{D}(u), q)
+\sum^{q}_{i=2}[R^{0}_{a+1}(q, d^{+}_{D}(u_i))-R^{0}_{a+1}(q'\\
&,d^{+}_{D'}(u_i))]+\sum^{d_{G}(v)}_{i=q+1}[R^{0}_{a+1}(p,d^{-}_{D}(u_i))-R^{0}_{a+1}(p',d^{-}_{D'}(u_i))]]\\
=& \frac{1}{2}[(d^{+}_{D}(u))^{a}+(q)^{a}
+(q-1)[(q)^{a}-(q-1)^{a}]]\\
\leq& \frac{1}{2}[1+(d_{G}(v))^{a}+(d_{G}(v)-1)[(d_{G}(v))^{a}-(d_{G}(v)-1)^{a}]]\\
=&\frac{1}{2}[1+(d_{G}(v))^{a+1}-(d_{G}(v)-1)^{a+1}]
\end{aligned}
\end{equation*}
with equality if and only if   $q=d_{G}(v)$.

\textbf{Case 2}. $vu\in A(D)$

  Similarly to  \textbf{Case 1}, we have
$$R^{0}_{a+1}(D)-R^{0}_{a+1}(D')\leq \frac{1}{2}[1+(d_{G}(v))^{a+1}-(d_{G}(v)-1)^{a+1}]$$
with equality if and only if   $p=d_{G}(v)$.

Hence, $$R^{0}_{a+1}(D)-R^{0}_{a+1}(D')\leq \frac{1}{2}[1+(d_{G}(v))^{a+1}-(d_{G}(v)-1)^{a+1}]$$
  with equality if and only if  $\max \{d_{D}^{+}(v), d_{D}^{-}(v)\}=d_{G}(v)$.
\end{proof}

Next, we will get the oriented unicyclic graphs  with the maximum of  the zeroth-order general Randi\'{c} index and the corresponding extreme graph, which will help to prove the main conclusions.
\begin{lemma}\label{lem8}
Let $G\in \mathcal{G}(n,1)$, $n\geq 3$, $D\in \mathcal{O}(G)$, $a\geq 1$. Then
$$R^{0}_{a+1}(D)\leq \frac{1}{2}[(n-1)^{a+1}+n-1+2^{a+1}]$$
  with equality if and only if $D\in  \mathcal{G}^{*}(n,1)$.
\end{lemma}
\begin{proof}
We proceed by applying induction on $n$.\\
\indent If $n=3$, then  $G=C_3$,  $D\in \mathcal{O}(C_3)$, it is straightforward to check that $R^{0}_{a+1}(D)\leq \frac{1}{2}[2^{a+2}+2]$ with equality if and only if $D\in \{G^{(1)}(3,1),G^{(2)}(3,1)\}=\mathcal{G}^{*}(3,1)$.

If $n=4$, then  $G=\{G_1, C_4\} $, where $G_1$ is shown in  Figure \ref{fig-2}. Let $f(x)=\frac{1}{2}[2^{x+1}+3+3^{x+1}]-2^{x+2}$. When $a\geq2$, $f(a)=\frac{1}{2}[2^{a+1}+3+3^{a+1}]-2^{a+2}=\frac{1}{2}[3+3^{a+1}-3\times 2^{a+1}]=\frac{1}{2}[3+27\times 3^{a-2}-24\times 2^{a-2}]> 0$;
When  $1\leq a\leq 2$, $f'(a)=\frac{1}{2}[3^{a+1}\times ln3-3\times 2^{a+1}\times ln2]>0$ by MATLAB (see Figure \ref{fig-12} of the Appendix A). $f(a)\geq f(1)=0$.
   From Lemma \ref{lem2}, if $D\in \mathcal{O}(C_4)$, we can get  $R^{0}_{a+1}(D)\leq \frac{1}{2}R^{0}_{a+1}(C_4)=2^{a+2}\leq\frac{1}{2}[2^{a+1}+3+3^{a+1}]$ with equality if and only if $a=1$ and $D$ is a sink-source orientation of $C_4$, which implies that $D\in \{G^{(3)}(4,1), G^{(4)}(4,1)\}$. By Lemma \ref{lem10}, if $D\in \mathcal{O}(G_1)$, $M_1(D)\leq \frac{1}{2}[2^{a+1}+3+3^{a+1}]$  with equality if and only if $D\in \{G^{(1)}(4,1),G^{(2)}(4,1)\}$.
Thus the result holds for $n=3,4$.

 If $n \geq 5$ and suppose that the lemma is true completely for all orientations of graphs in $\mathcal{G}(n-1,1)$.
  Let $G\in \mathcal{G}(n,1)$.

 \textbf{Case 1}.  $G=C_n$.

  If $G=C_n$,  $D\in \mathcal{O}(C_n)$,
 let $f(n)=\frac{1}{2}[2^{a+1}+n-1+(n-1)^{a+1}]-n\times 2^{a},$
 then  $f'(n)=\frac{1}{2}[1+(a+1)(n-1)^{a}]-2^{a}>0$, and
    $f(n)\geq f(5)=\frac{1}{2}[2^{a+1}+4+4^{a+1}]-5\times 2^{a}>0.$

    From Lemma \ref{lem2},  we can get  that $R^{0}_{a+1}(D)\leq \frac{R^{0}_{a+1}(C_n)}{2}=n\times 2^{a}< \frac{1}{2}[(n-1)^{a+1}+n-1+2^{a+1}]$. The lemma holds.

\textbf{ Case 2}.  $G\neq C_n$.

 Let $G'=G-\{u\}$, where $d_{G}(u)=1$, then $G'\in \mathcal{G}(n-1,1)$.  Let $v\in V(G)$, $N_{G}(v)=\{u=u_1,u_2,\cdot\cdot\cdot,u_{s}\}$, $1\leq s=d_{G}(v)\leq n-1$,   $D'\in\mathcal{O}(G')$ such that  $A(D')\cap A(D)=A(D')$.

By the induction hypothesis,  $$R^{0}_{a+1}(D')\leq  \frac{1}{2}[(n-2)^{a+1}+n-2+2^{a+1}].$$


From Lemma \ref{lem102}, we can obtain  that
\begin{equation*}
\begin{aligned}
R^{0}_{a+1}(D) & \leqslant R^{0}_{a+1}(D')+\frac{1}{2}[1-(d_{G}(v)-1)^{a+1}+(d_{G}(v))^{a+1}]\\
& \leqslant \frac{1}{2}[2^{a+1}+n-2+(n-2)^{a+1}]+\frac{1}{2}[1+(n-1)^{a+1}-(n-2)^{a+1}]\\
& = \frac{1}{2}[(n-1)^{a+1}+n-1+2^{a+1}]
\end{aligned}
\end{equation*}
 with equality if and only if $R^{0}_{a+1}(D')= \frac{1}{2}[n-2+(n-2)^{a+1}+2^{a+1}]$,  and $\{d_{D}^{+}(u)=d_{G}(u),d_{D}^{-}(v)=d_{G}(v)=n-1\}$ or $\{d_{D}^{+}(v)=d_{G}(v)=n-1,d_{D}^{-}(u)=d_{G}(u)\}$, which implies that $D\in \mathcal{G}^{*}(n,1)$.

 The proof of the lemma is completed.
\end{proof}

\begin{lemma}\label{lem9}
Let $D\in \mathcal{O}(G_2)$, $G_2\in \mathcal{G}(5,2)$, where $G_2$ is depicted in  Figure \ref{fig-3}, $a\geq 1$. Then $$R^{0}_{a+1}(D)\leq 2\times 4^a+2^{a+1}+2$$
 with equality if and only if  $D\in \mathcal{G}^{*}(5,2)$.
\end{lemma}

\begin{figure}[ht]
\begin{center}
  \includegraphics[width=11cm,height=8cm]{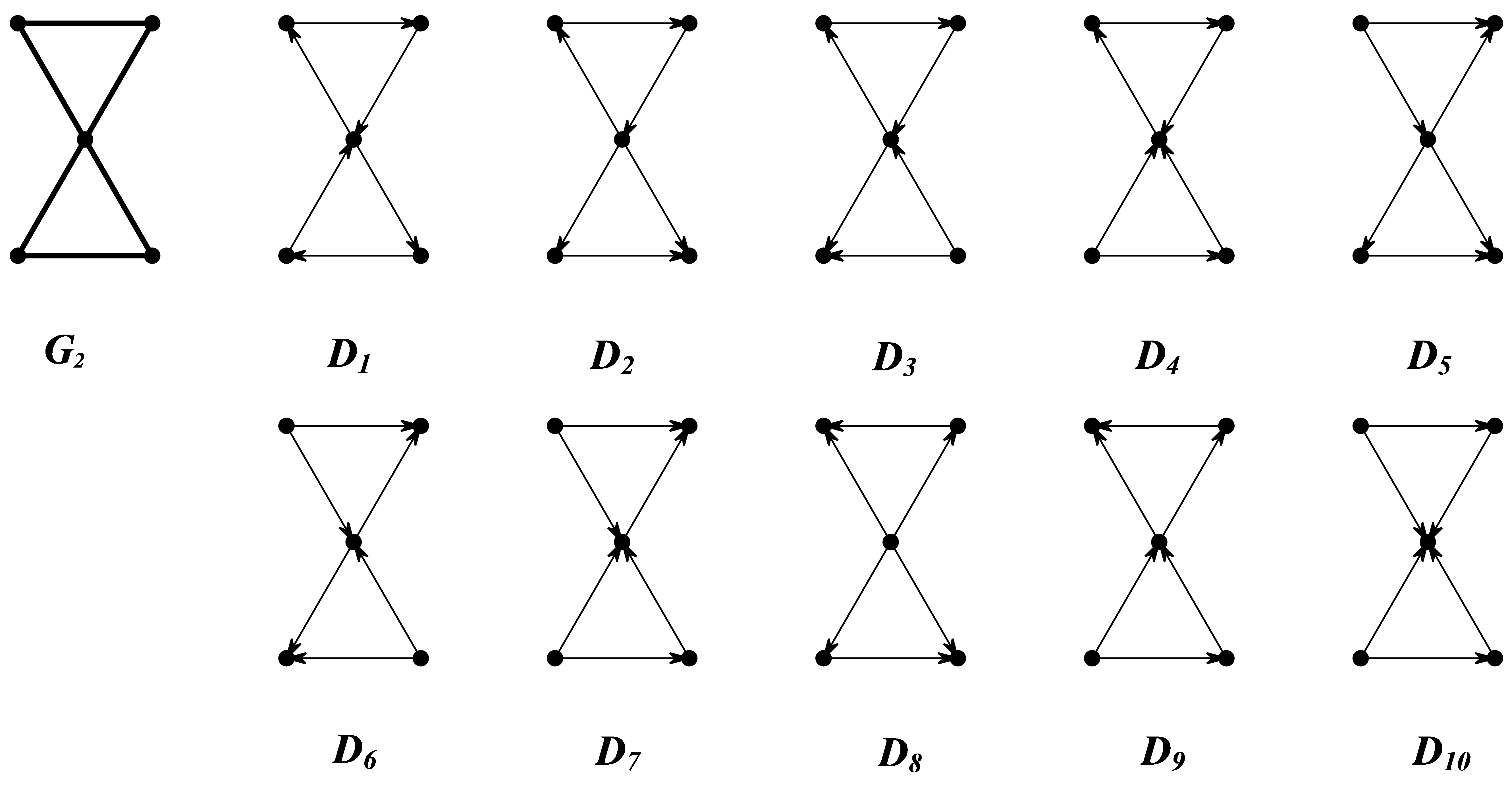}
 \end{center}
\vskip -0.5cm
\caption{$G_2$ and its orientations  }\label{fig-3}
\end{figure}

\begin{proof}
 It is straightforward to check that $\mathcal{O}(G_2)=\{D_1,D_2,\cdot\cdot\cdot, D_{10}\}$, where $D_1,D_2,\cdot\cdot\cdot, D_{10}$ are depicted in Figure \ref{fig-3}.
By   direct computing, we obtain

$R^{0}_{a+1}(D_1)=\frac{1}{2}(8+2^{a+2})$

$R^{0}_{a+1}(D_4)=R^{0}_{a+1}(D_2)=\frac{1}{2}(7+3^{a+1}+2^{a+1})$

$R^{0}_{a+1}(D_9)=R^{0}_{a+1}(D_3)=\frac{1}{2}(4+2^{a+3})$

$R^{0}_{a+1}(D_7)=R^{0}_{a+1}(D_5)=\frac{1}{2}(3+6\times 2^{a}+3^{a+1})$

$R^{0}_{a+1}(D_6)=\frac{1}{2}(12\times 2^{a})$

$R^{0}_{a+1}(D_{10})=R^{0}_{a+1}(D_8)=2\times 4^a+2^{a+1}+2$.

As the fact that

$R^{0}_{a+1}(D_8)-R^{0}_{a+1}(D_1)=2\times 4^a+2^{a+1}+2-\frac{1}{2}(8+2^{a+2})=2\times 4^a-2>0$

$R^{0}_{a+1}(D_8)-R^{0}_{a+1}(D_2)=2\times 4^a+2^{a+1}+2-\frac{1}{2}(7+3^{a+1}+2^{a+1})=2\times 4^a-\frac{3}{2}\times 3^a+2^a-1.5>0$

$R^{0}_{a+1}(D_8)-R^{0}_{a+1}(D_3)=2\times 4^a+2^{a+1}+2-\frac{1}{2}(4+2^{a+3})=2\times 4^a-2^{a+1}>0$

$R^{0}_{a+1}(D_8)-R^{0}_{a+1}(D_5)=2\times 4^a+2^{a+1}+2-\frac{1}{2}(3+6\times 2^{a}+3^{a+1})=2\times 4^a-\frac{3}{2}\times 3^{a}-2^a+\frac{1}{2}=8\times 4^{a-1}-\frac{9}{2}\times 3^{a-1}-2\times 2^{a-1}+\frac{1}{2}>0$

$R^{0}_{a+1}(D_8)-R^{0}_{a+1}(D_6)=2\times 4^a+2^{a+1}+2-\frac{1}{2}(12\times 2^{a})=(2^{a+1}+2-6)2^a+2>0$

It is easy to get  that $$R^{0}_{a+1}(D)\leq 2\times 4^a+2^{a+1}+2$$
 with equality if and only if  $D\in \{D_{10},D_8\}=\mathcal{G}^{*}(5,2)$.

\end{proof}

\section{Main Result}
 We will solve the problem which is to find  the maximum    zeroth-order general Randi\'{c} index over the set of all oriented cacti in terms of the order in the section.
The main result is as follows.

\begin{theorem}\label{the11}
Let $G\in \mathcal{G}(n,r)$, $D\in \mathcal{O}(G)$, $a\geq 1$. Then
$$R^{0}_{a+1}(D)\leq \frac{1}{2}[(n-1)^{a+1}+n-1+2r\times 2^{a}]$$
 with equality if and only if $D\in \mathcal{G}^{*}(n,r)$.
\end{theorem}
 \begin{proof} We proceed by applying induction on $n$ and $r$.  By Lemma \ref{lem7} and Lemma \ref{lem8},  the result holds  for $r=0$ and 1. If $r\geq 2$ , then $n\geq 5$. If $n=5$,  $|\mathcal{G}(5,2)|=1$ and $G_{2}\in \mathcal{G}(5,2)$ (see Figure \ref{fig-3}), then the result holds  by Lemma \ref{lem9}.

 If $n\geq 6$ and $r\geq 2$.  We distinguish the following two cases according to $\delta(G)$.

\textbf{Case 1}. $\delta(G)=1$.

Let $u,v\in V(G)$ and $d_{G}(u)=1$, $v\in N_{G}(u)$,
$G'=G-u$. Obviously, $G'\in \mathcal{G}(n-1,r)$, $d_{G}(v)\leq n-1$.

Let  $D'\in\mathcal{O}(G')$ such that  $A(D')\cap A(D)=A(D')$.
By the induction hypothesis,  $$R^{0}_{a+1}(D')\leq \frac{1}{2}[(n-2)^{a+1}+n-2+2r\times 2^{a}].$$


From Lemma \ref{lem102}, we have
\begin{equation*}
\begin{aligned}
R^{0}_{a+1}(D)&\leq R^{0}_{a+1}(D')+\frac{1}{2}[1+(d_{G}(v))^{a+1}-(d_{G}(v)-1)^{a+1}]\\
 &\leqslant  \frac{1}{2}[(n-2)^{a+1}+n-2+2r\times 2^{a}+1+(n-1)^{a+1}-(n-2)^{a+1}]\\
& = \frac{1}{2}[(n-1)^{a+1}+n-1+2r\times 2^{a}]
\end{aligned}
\end{equation*}
  with equality if and only if $R^{0}_{a+1}(D')= \frac{1}{2}[(n-2)^{a+1}+n-2+2r\times 2^{a}]$, $\max \{d_{D}^{+}(v), d_{D}^{-}(v)\}=d_{G}(v)=n-1$, which imply the fact that $D\in \mathcal{G}^{*}(n,r)$. The result holds.

\textbf{Case 2}. $\delta(G)\geq 2$.

 It is obvious that there exists  $u_0u_1\in E(G)$ in a terminal block of $G$ such that $d_{G}(u_0)=d_{G}(u_1)=2$, $N_{G}(u_0)=\{u_1,u_2\}$ and  $d_{G}(u_2)=d\geq 3$. We will take two subcases into consideration.

\textbf{Subcase 2.1}. $u_1u_2\notin E(G)$.

Let $G'=G-u_0+u_1u_2$ and $G''=G-u_0$. Obviously, $G'\in \mathcal{G}(n-1,r)$ and $d\leq n-2$. Let  $D'\in\mathcal{O}(G')$, $D''\in\mathcal{O}(G'')$  such that  $A(D')\cap A(D)=A(D'')$. By the induction hypothesis,  we can obtain the inequality that
$$R^{0}_{a+1}(D')\leq \frac{1}{2}[(n-2)^{a+1}+n-2+2r\times 2^{a}].$$ From  Lemma \ref{lem4},  we  obtain

\begin{equation*}
\begin{aligned}
R^{0}_{a+1}(D) & \leq R^{0}_{a+1}(D')+\frac{1}{2}[2^{a+1}+d^{a+1}-(d-1)^{a+1}-1]\\
& \leq \frac{1}{2}[(n-2)^{a+1}+n-2+2r\times 2^{a}+2^{a+1}+d^{a+1}-(d-1)^{a+1}-1]\\
& \leq \frac{1}{2}[(n-2)^{a+1}+n-2+2r\times 2^{a}+2^{a+1}+(n-2)^{a+1}-(n-3)^{a+1}-1]\\
&=\frac{1}{2}[2(n-2)^{a+1}+n-3+(r+1)\times 2^{a+1}-(n-3)^{a+1}]
\end{aligned}
\end{equation*}

Let $f(n)=\frac{1}{2}[(n-1)^{a+1}+n-1+2r\times 2^{a}]-\frac{1}{2}[2(n-2)^{a+1}+n-3+(r+1)\times 2^{a+1}-(n-3)^{a+1}]$ and $g(x)=x^a$ $(x>0)$, we have $g'(x)=ax^{a-1}$  and $g''(x)=a(a-1)x^{a-2}\geq 0$. Hence
$f'(n)=\frac{1}{2}[(a+1)(n-1)^{a}+1]-\frac{1}{2}[2(a+1)(n-2)^{a}+1-(a+1)(n-3)^{a}]=\frac{1}{2}(a+1)[(n-1)^{a}-2(n-2)^{a}+(n-3)^{a}]\geq 0$. It is obvious that $\min\limits_{n\geq 6}f(n)=f(6)$.

If $a\geq 3$, then
$f(6)=\frac{1}{2}[5^{a+1}+5+2r\times 2^{a}]-\frac{1}{2}[2\times 4^{a+1}+3+(r+1)\times 2^{a+1}-3^{a+1}]=\frac{1}{2}[5^{a+1}+2-2\times 4^{a+1}+3^{a+1}-2^{a+1}]=\frac{1}{2}[625\times 5^{a-3}+2-512\times 4^{a-3}+3^{a+1}-2^{a+1}]>0$.

If $a=1$, then
$f(6)=\frac{1}{2}[5^{1+1}+5+2r\times 2^{1}]-\frac{1}{2}[2\times 4^{1+1}+3+(r+1)\times 2^{1+1}-3^{1+1}]=0$.

 If $1\leq a\leq 3$, let $h(a)=f(6)$, then $h'(a)=\frac{1}{2}[5^{a+1}\times ln5-2\times 4^{a+1}\times ln4+3^{a+1}\times ln3-2^{a+1}\times ln2]>0$ by direct computing of MATLAB (see Figure \ref{fig-11} of the Appendix A).
 Hence, $f(6)=h(a)\geq h(1)= 0$.

Consequently, $$R^{0}_{a+1}(D)\leq \frac{1}{2}[(n-1)^{a+1}+n-1+2r\times 2^{a}]$$ with equality if and only if
$a=1$, $D'\in \mathcal{G}^{*}(n-1,r)$, $d=n-2$, $u_1u_0, u_2u_0\in A(D), u_1u_2\in A(D'), d^{+}_{D}(u_2)=d_{G}(u_2)$; or  $u_0u_1, u_0u_2\in A(D), u_2u_1\in A(D'), d^{-}_{D}(u_2)=d_{G}(u_2)$; or $u_0u_1, u_2u_0\in A(D), u_1u_2\in A(D'), d^{-}_{D}(u_1)=d_{G}(u_1), d^{+}_{D}(u_2)=d_{G}(u_2)$; or $u_1u_0, u_0u_2\in A(D), u_2u_1\in A(D'), d^{+}_{D}(u_1)=d_{G}(u_1), d^{-}_{D}(u_2)=d_{G}(u_2)$.
But if $D'\in \mathcal{G}^{*}(n-1,r)$, $u_1u_2\in A(D')$, we have $d^{-}_{D'}(u_2)=d_{G'}(u_2)$. It is  a contradiction with  $d^{+}_{D}(u_2)=d_{G}(u_2)$.
If $D'\in \mathcal{G}^{*}(n-1,r)$, $u_2u_1\in A(D')$, we have $d^{+}_{D'}(u_2)=d_{G'}(u_2)$. It is  a contradiction with  $d^{-}_{D}(u_2)=d_{G}(u_2)$.
Hence,
 $$R^{0}_{a+1}(D)< \frac{1}{2}[(n-1)^{a+1}+n-1+2r\times 2^{a}].$$

\textbf{Subcase 2.2}. $u_1u_2\in E(G)$.

Let $G'=G-u_0-u_1$. Obviously, $G'\in \mathcal{G}(n-2,r-1)$ and $d\leq n-1$.  Let $D'\in\mathcal{O}(G')$ such that  $A(D')\cap A(D)=A(D')$.  By the induction hypothesis,
$$R^{0}_{a+1}(D')\leq \frac{1}{2}[(n-3)^{a+1}+n-3+2(r-1)\times 2^{a}]$$
 By Lemma \ref{lem5},  we can obtain

\begin{equation*}
\begin{aligned}
R^{0}_{a+1}(D) & \leqslant R^{0}_{a+1}(D')+\frac{1}{2}[2^{a+1}+2+d^{a+1}-(d-2)^{a+1}]\\
& \leq \frac{1}{2}[(n-3)^{a+1}+n-3+2(r-1)\times 2^{a}+2^{a+1}+2+d^{a+1}-(d-2)^{a+1}]\\
& \leq \frac{1}{2}[(n-3)^{a+1}+n-3+2(r-1)\times 2^{a}+2^{a+1}+2+(n-1)^{a+1}-(n-3)^{a+1}]\\
&=\frac{1}{2}[(n-1)^{a+1}+n-1+2r\times 2^{a}]
\end{aligned}
\end{equation*}
with equality if and only if $R^{0}_{a+1}(D')= \frac{1}{2}[(n-3)^{a+1}+n-3+2(r-1)\times 2^{a}]$, $d=n-1$,
or equivalently, $D'\in \mathcal{G}^{*}(n-2,r-1)$ and $\{d_{D}^{-}(u_2)=d_{G}(u_2)=n-1\}$ or $\{d_{D}^{+}(u_2)=d_{G}(u_2)=n-1\}$, i.e., $D\in \mathcal{G}^{*}(n,r)$.   The result holds for $n$.
\end{proof}

\textbf{Acknowledgment}.
 This work is supported by the Hunan Provincial Natural Science Foundation of China (2020JJ4423), the Department of Education of Hunan Province (19A318) and the National Natural Science Foundation of China (11971164).

\textbf{Appendix A}
\begin{figure}[ht]
\begin{center}
  \includegraphics[width=10cm,height=6cm]{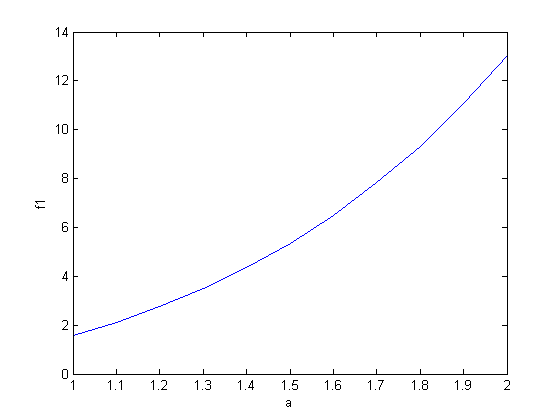}
 \end{center}
\vskip -0.5cm
\caption{$f1=2f'(a)=3^{a+1}\times ln3-3\times 2^{a+1}\times ln2>0$ $(1\leq a\leq 2)$ }\label{fig-12}
\end{figure}
\begin{figure}[ht]
\begin{center}
  \includegraphics[width=11cm,height=8cm]{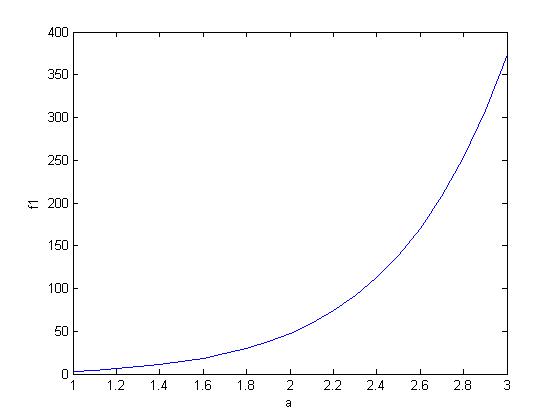}
 \end{center}
\vskip -0.5cm
\caption{$f1=2h'(a)=5^{a+1}\times ln5-2\times 4^{a+1}\times ln4+3^{a+1}\times ln3-2^{a+1}\times ln2>0$ $(1\leq a\leq 3)$}\label{fig-11}
\end{figure}


\begin{thebibliography}{00}

\vspace{-7pt}\bibitem{1} M. V. Diudea, QSPR/QSAR Studies by Molecular Descriptors, Nova Science Publishers, New York, NY, USA, 2001.

\vspace{-7pt}\bibitem{2} A. T. Balaban, J. Devillers, Eds., Topological Indices and Related Descriptors in QSAR and QSPAR, CRC Press,
Florida, FL, USA, 2014.

\vspace{-7pt}\bibitem{3}  X. Li, J. Zheng, A unified approach to the extremal trees for different indices, MATCH  Communications in Mathematical and in Computer Chemistry,  54  (2005) 195-208.

\vspace{-7pt}\bibitem{4} M. Liu, B. Liu, Some properties of the first general Zagreb index, Australasian Journal of Combinatorics,  47 (2010)  285.

\vspace{-7pt}\bibitem{5} J. M. Rodr\'{i}guez, J. L. S\'{a}nchez,  J. M. Sigarreta, CMMSE-on the first general Zagreb index, Journal of Mathematical
Chemistry,  56  (2018) 1849-1864.

\vspace{-7pt}\bibitem{6} H. M. Awais, M. Javaid, A. Ali, First general Zagreb index of generalized F-sum graphs, Discrete Dynamics in Nature
and Society,  2020 (2020) 1-16.

\vspace{-7pt}\bibitem{7}X. F. Pan, H. Liu, M. Liu, Sharp bounds on the zeroth-order general Randi\'{c} index of unicyclic graphs with given
 diameter, Applied Mathematics Letters,  24  (2011) 687-691.

\vspace{-7pt}\bibitem{8} M. K. Jamil, I. Tomescu, Zeroth-order general Randi\'{c} index of k-generalized quasi trees, (2018), https://arxiv.org/
abs/1801.03885.

\vspace{-7pt}\bibitem{9} M. Azari, A. Iranmanesh, Generalized Zagreb index of graphs, Studia Universitatis Babes-Bolyai
Chemia, 56  (2011) 59-70.

\vspace{-7pt}\bibitem{10} J. B. Liu, S. Javed, M. Javaid,  K. Shabbir, Computing first general Zagreb index of operations on graphs, IEEE access, 7 (2019) 47494-47502.

\vspace{-7pt}\bibitem{11} L. Bedratyuk, O. Savenko, The star sequence and the
general first Zagreb index, (2017), MATCH  Communications in Mathematical and in Computer Chemistry, 79  (2018) 407-414.







\vspace{-7pt}\bibitem{JM212}J. Monsalve, J. Rada, Sharp
upper and lower bounds of VDB topological indices of digraphs,
Symmetry, 13  (2021) 1903.

\vspace{-7pt}\bibitem{JM211} J. Monsalve, J. Rada, Vertex-degree based topological indices of digraphs, Discrete Applied Mathematics, 295 (2021)  13-24.

\vspace{-7pt}\bibitem{JY22}J. Yang, H. Deng, maximum first Zagreb index of orientations of unicyclic graphs with given matching number, Applied Mathematics and Computation, 427 (2022) 127131.

\vspace{-7pt}\bibitem{HD22} H. Deng, Z. Tang, J. Yang, J. Yang, M. You, On the vertex-degree based invariants of digraphs, Discrete  Mathematics Letter, 9 (2022)  2-9.

\vspace{-7pt}\bibitem{JM19} J. Monsalve, J. Rada, Oriented bipartite graphs with minimal trace norm, Linear Multilinear Algebra, 67  (2019) 1121-1131.





\end{thebibliography}
\end{document}